\documentclass{siamart1116}
\usepackage{amsmath}
\usepackage{amssymb}
\usepackage{graphicx} 
\usepackage{bm}
\usepackage{color}

\usepackage{geometry}
\usepackage{dcolumn}
\usepackage{color} 

\usepackage{lineno}
\newcommand{\bs}{\backslash}
\newcommand{\R}{\mathbb{R}}

\newcommand{\Pn}{\mathcal{P}}
\newcommand{\F}{\mathcal{F}}

\newcommand{\toz}[1]{\quad\mbox{as $#1 \to 0$}}
\newcommand{\ub}[1]{^{(#1)}}
\newcommand{\K}{\mathcal{K}}
\newcommand{\subsc}[1]{_{\mbox{\scriptsize #1}}}
\usepackage{tabularx}
\newcommand{\ul}[1]{\underline{#1}}

\begin{document}

\title{Minimal transport networks with general boundary conditions}

\author{Shyr-Shea Chang%
\thanks{Dept. of Mathematics, University of California Los Angeles, Los Angeles, CA 90095, USA.}%
\and
Marcus Roper%
\footnotemark[1]
\thanks{Dept. of Biomathematics, University of California Los Angeles, Los Angeles, CA 90095, USA.}
}

\maketitle

\begin{abstract}
Vascular networks are used across the kingdoms of life to transport fluids, nutrients and cellular material. A popular unifying idea for understanding the diversity and constraints of these networks is that the conduits making up the network are organized to optimize dissipation or other functions within the network. However the general principles governing the optimal networks remain unknown. In particular Durand \cite{durand2007structure} showed that under Neumann boundary conditions networks, that minimize dissipation should be trees. Yet many real transport networks, including capillary beds, are not simply connected. Previously multiconnectedness in a network has been assumed to provide evidence that the network is not simply minimizing dissipation. Here we show that if the boundary conditions on the flows within the network are enlarged to include physical reasonable Neumann and Dirichlet boundary conditions (i.e. constraints on either pressure or flow) then minimally dissipative networks need not be trees. To get to this result we show that two methods of producing optimal networks, namely enforcing constraints via Lagrange multipliers or via penalty methods, are equivalent for tree networks.
\end{abstract}

\section{Introduction} \label{sec:intro}

Organisms across the kingdoms of life; including plants, animals, fungi, and watermolds rely on vascular networks to transport fluids, nutrients or cellular materials \cite{reece2011campbell}. In vertebrate animals, a cardiovascular network transports oxygenated blood from the heart to tissues throughout the body, and returns waste gases to the heart and lungs. Distruption of this network even at the level of finest vessels, including the systemic microvessel degradation associated with diabetes mellitus, or acute damage associated with traumatic brain injury, has long term irreparable health consequences. Accordingly parallel experimental efforts have targeted the same goal of complete mapping of microvascular networks \cite{blinder2013cortical,wu20143d,kelch2015organ}. Yet interpreting these data streams is held back by lack of information on the organizing principles underlying the mapped networks.

One principle that has been used to dissect these networks is Murray's law (Murray, 1926 \cite{murray1926physiological}). Murray's law states that if a network made up of hydraulic conduits minimizes a total cost made up of the sum of the total dissipation and of the material used to build the network, then the radius of each conduit within the network is proportional to the cube root of the flow that it carries. Murray's law has been verified by studies on plant and mammalian vascular networks (\cite{sherman1981connecting,mcculloh2003water,taber2001investigating}, but also see \cite{sherman1981connecting} for a discussion of exemplar networks that do not apparently obey Murray's law). The derivation of this result draws several assumptions that we will systematically analyze in this paper, so we present a brief derivation here. Consider a cylindrical tube with radius $r$ and length $\ell$ with a flow $f$ going through. By flow we mean that a volume $f$ of fluid (e.g. blood) passes through each cross-section of the network in unit time. In appropriate units, the energy cost of maintaining the vessel can be written as

\begin{equation}
E = D + ar^2\ell = f^2 R + ar^2\ell
\end{equation}
where $D = f^2 R$ is the dissipation, $R$ is the hydraulic resistance, and $a$ is the energy cost for maintaining unit volume of blood. Under Hagen-Poiseuille's law $R = \frac{8\mu \ell}{\pi r^4}$ where $\mu$ is the viscosity of the blood. Suppose $r$ is tuned such that the energy cost is minimized under fixed amount of inflow $f$. Then the derivative of $E$ over $r$ should vanish, i.e.

\begin{equation}\frac{dE}{dr} = 0\Rightarrow -\frac{32 f^2 \mu \ell}{\pi r^5} + 2ar\ell = 0 \Rightarrow f = \sqrt{\frac{a \pi}{16\mu}} r^3
\end{equation}
and hence, as claimed, $r\;\propto\; f^\frac{1}{3}$.

A key part of this derivation is that changing the radius of the vessel does not affect the flows passing through it. In other words, flows and radii can be treated as independent variables. However the flows within a network generally depend on the conductances within the network -- so changing radii of vessels within the network may alter the flows. Accordingly it is not obvious that when the feedback between vessel radius and flow is considered, i.e. when conduits are considered assembled within a network, Murray's law will continue to hold, or that a dissipation minimizing network configuration actually exists.

Durand \cite{durand2007structure} studied optimization of dissipation on networks in which multiple sources were linked to multiple edges with arbitrarily complex network of edges and vertices. A prior set of edges can be assigned (potentially including straightline paths between every pair of sources and or sinks), and one searches for the network that uses some, but not necessarily all of the prior edges, that minimizes the total dissipation for a prescribed material cost. This approach, in which material is prescribed as a holonomic constraint and a minimally dissipative network is sought consistent with this constraint, is not obviously equivalent, in the sense of producing the same family of optimal networks, as Murray's approach, which we may view as a penalty function method for optimizing dissipation under material cost. But it has been adopted in many recent works on optimal networks \cite{durand2007structure,bohn2007structure,katifori2010damage}. Durand showed that any network that solves this optimization problem must be simply connected. However the proof given in \cite{durand2007structure} leaves unclear what combinations of boundary conditions are allowed in such networks. Significantly we can quickly see that for some combinations of boundary conditions the minimally dissipative network is not simply connected. For example consider a square network made of 4 edges and 4 vertices, all of which have pressure specified. Vertices and edges are numbered as shown in Figure \ref{fig:diss_loop}. In this network flows can be determined locally, i.e. the flow on one link does not depend on flows on others. Specifically

\begin{figure}[t]
\begin{center}

\includegraphics[width = 7cm]{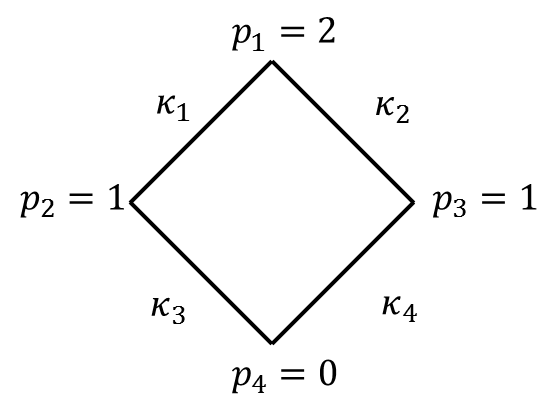}
\caption{A non-tree like minimal dissipation network.}
\label{fig:diss_loop}

\end{center}
\end{figure}

\begin{equation}
Q_1 = \kappa_1, Q_2 = \kappa_2, Q_3 = \kappa_3, Q_4 = \kappa_4.
\end{equation}
The total dissipation within the network is

\begin{equation}
D = \sum_{i=1}^4 \kappa_i.
\end{equation}
We follow Durand by specifying the total material available to build the network. Since all edges have the same length this constraint takes the form $\sum_i r^2_i = \mbox{const}$, where $R_i$ is the radius of edge $i$. Now since by the Hagen-Poiseuille law $\kappa_i\propto r^4_i$, we may equivalently write the constraint in the form:

\begin{equation}
K^\frac{1}{2} = \sum_{i=1}^4 \kappa^\frac{1}{2}_i,
\end{equation}
for some $K>0$. To minimize dissipation under the material constraint we write the dissipation in the network and add a Lagrange multiplier to enforce the material constraint:

\begin{equation}
\Theta = \sum_{i=1}^4 \kappa_i + \lambda (\sum_{i=1}^4 \kappa^\frac{1}{2}_i - K^\frac{1}{2}).
\end{equation}
We find the optimal conductances within the network by setting equal to 0 each of the partial derivatives of $\Theta$ with respect to the variables $\{\kappa_i\}$ in the form:

\begin{equation}
0 = \frac{\partial\Theta}{\partial\kappa_i} = 1 +\frac{\lambda}{2}\kappa^{-\frac{1}{2}}_i \Rightarrow \kappa_i = \frac{\lambda^2}{4} \qquad \forall 1\leq i \leq 4.
\end{equation}
The Lagrange multiplier $\lambda$ can be determined from the material constraint:

\begin{equation}
K^\frac{1}{2} = \sum_{i=1}^4 \kappa^\frac{1}{2}_i = 2\lambda \Rightarrow \lambda = \frac{K^\frac{1}{2}}{2}.
\end{equation}
We have therefore identified a candidate local extremum with $\kappa_i > 0 \; \forall 1\leq i \leq 4$. but this local extremum might not be the global minimizer. The set on which we need to minimize the dissipation, i.e. $\{(\kappa_1,\kappa_2,\kappa_3,\kappa_4)| \sum_{i=1}^4 \kappa_i^\frac{1}{2} = K^\frac{1}{2}\}$, is compact, so the global minimum must be attained either at the local extremum, or on one of the set boundaries $\kappa_i = 0$ for some $1\leq i \leq 4$. To analyze the dissipation on domain boundaries we can simply assume that $n\leq 4$ conductances are positive and recalculate $\lambda$ in the same fashion:

\begin{equation}
K^\frac{1}{2} = \sum_{i=1}^4 \kappa^\frac{1}{2}_i = \frac{n}{2}\lambda \Rightarrow \lambda = \frac{2K^\frac{1}{2}}{n}.
\end{equation}

\noindent Now we can calculate the dissipation and see which $n$ gives the lowest dissipation (let $\K \subseteq {1,2,3,4}$ be the set of positive conductances so $|\K| = n$):

\begin{equation}
D = \sum_{i\in \K} \frac{1}{4}\lambda^2 =   \sum_{i\in \K} \frac{K}{n^2} = \frac{K}{n}
\end{equation}

\noindent so $n = 4$ indeed results in minimal dissipation network; consisting of a single loop through all four vertices. Note additionally that, on this prior network, treating material costs as holonomic constraint or penalty function does not produce equivalent results. Indeed the sum of dissipation and material consts is trivially minimized in a network in which all edges have been eliminated.

The above example shows that minimizing dissipation on a network with multiple pressure boundary conditions produces a multiply connected, i.e. non-tree network. The relevance of the example network shown in Figure \ref{fig:diss_loop} to real biological transport network design may seem unconvincing; however even quite simple networks commonly used as models for biological transport can exhibit non-equivalent optima under the different formulations for material costs. To see how substantial the difference can be we can consider a simple network comprising two edges (Fig.~\ref{fig:contraint_penalty}A) and minimizing

\begin{figure}[t]
\begin{center}

\includegraphics[width = 15cm]{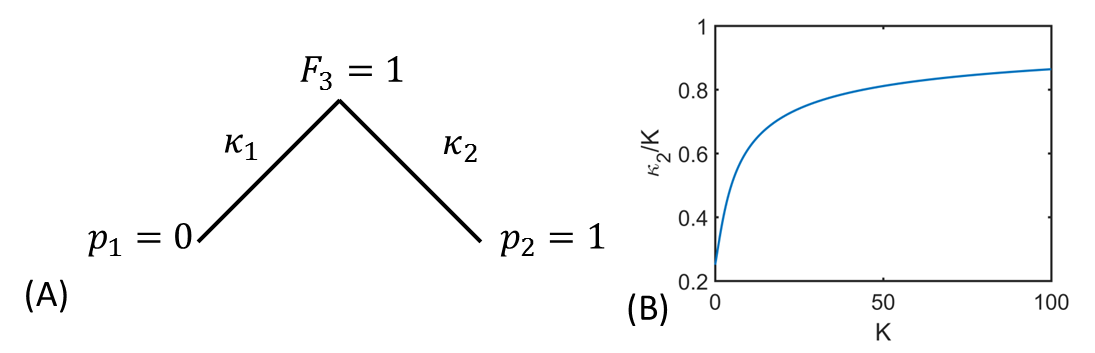}
\caption{The different formulations of imposing material as constraint or penalty function affect the optimal network for the same target function. (A) A network in which a vertex with prescribed inflow, $F_3=1$, is connected to two vertices on which pressures are prescribed. (B) The asymmetry of the network increases as the total prescribed material $K$ increases, as predicted by the asymptotic analysis in Section \ref{sec:intro}.}
\label{fig:contraint_penalty}

\end{center}
\end{figure}

\begin{equation}
f = \sum_{i=1}^2 (Q_i-\frac{1}{2})^2.
\end{equation}
This target function is inspired by our own studies of flow in microvascular networks, which have shown that uniform partitioning of flows in microvessels is prioritized over transport costs \cite{chang2015optimal}. By minimizing $f$ we are uniformizing the flows going through the edges to the pressure vertices. We compare the solutions from following either of our optimization approaches. First we treat the material cost as a penalty function, i.e. follow Murray's formulation, and minimize

\begin{equation}
\Theta = \sum_{i=1}^2 (Q_i-\frac{1}{2})^2 + a\sum_{i=1}^2 \kappa^\frac{1}{2}_i.
\end{equation}
The pressure at the flow vertex is determined by Kirchhoff's first law, which states that the flows along the two edges must sum to the inflow at vertex 3, i.e.:

\begin{equation}
p_3 \kappa_1 + (p_3-1)\kappa_2 = 1 \Rightarrow p_3 = \frac{1+\kappa_2}{\kappa_1 + \kappa_2}
\label{eq:two_edge_Kirchhoff}
\end{equation}
The total cost function $\Theta$ can be rewritten, after $p_3$ is solved for by Equation \ref{eq:two_edge_Kirchhoff}, as

\begin{equation}
\Theta = \left( \frac{(1+\kappa_2)\kappa_1}{\kappa_1+\kappa_2} - \frac{1}{2}\right)^2 + \left(\frac{(1-\kappa_1)\kappa_2}{\kappa_1+\kappa_2} - \frac{1}{2}\right)^2 + a(\kappa^\frac{1}{2}_1 + \kappa_2^\frac{1}{2}).
\end{equation}
We will show that $\Theta$ does not have a minimizer. First notice $(\kappa_1,\kappa_2) = (0,0)$ is not allowed since $p_3$ cannot be determined in this case. The minimum value of $\Theta$ is zero, and $(\kappa_1,\kappa_2) = (0,0)$ is the only configuration of the network that might achieve this value since otherwise $\kappa^\frac{1}{2}_1 + \kappa_2^\frac{1}{2}>0$. It suffices to show that we can find networks with $\Theta >0$ arbitrarily close to zero. If we let $\kappa_1=\kappa_2 = \epsilon>0$ then

\begin{equation}
\Theta = \frac{\epsilon^2}{2} + 2a\epsilon^\frac{1}{2}\to 0\qquad\toz{\epsilon}
\end{equation}
and we showed that $\Theta$ does not have a minimizer. On the other hand if we impose the total material as a constraint we have

\begin{equation}
\Theta = ( \frac{(1+\kappa_2)\kappa_1}{\kappa_1+\kappa_2} - \frac{1}{2})^2 + (\frac{(1-\kappa_1)\kappa_2}{\kappa_1+\kappa_2} - \frac{1}{2})^2
\end{equation}
where

\begin{equation}
\kappa_1^\frac{1}{2}+\kappa_2^\frac{1}{2} = K^\frac{1}{2}
\end{equation}
with a predetermined total material $K$. A minimum will happen if $Q_1 = Q_2 = \frac{1}{2}$ with the material constraint satisfied, so we may as well start from the equation

\begin{equation}
\frac{\kappa_1+\kappa_1\kappa_2}{\kappa_1+\kappa_2}=\frac{1}{2} \Rightarrow \kappa_1 = \frac{\kappa_2}{1+2\kappa_2}
\end{equation}
(the other equation is redundant since $Q_1+Q_2=1$). The material constraint then reads

\begin{equation}
[1+(1+2\kappa_2)^\frac{1}{2}]\frac{\kappa^\frac{1}{2}_2}{(1+2\kappa_2)^\frac{1}{2}} = K^\frac{1}{2}.
\end{equation}
This equation does not admit an analytical solution, but since the left hand side is monotonically increasing with $\kappa_2$ and can take any value between 0 and $\infty$, it can be solved for any finite $K>0$. In particular asymptotic solutions can be obtained as $K\to 0^+$ and as $K\to \infty$. When $K \ll 1$ we have $\kappa_2 \leq K \ll 1$ so

\begin{equation}
K^\frac{1}{2} = [1+(1+2\kappa_2)^\frac{1}{2}]\frac{\kappa^\frac{1}{2}_2}{(1+2\kappa_2)^\frac{1}{2}} \sim 2\kappa^\frac{1}{2}_2 \Rightarrow \kappa_1=\kappa_2 = \frac{K}{4}.
\end{equation}
In the case of $K\gg 1$ if we assume $\kappa_2\gg 1$ we can obtain

\begin{equation}
K^\frac{1}{2} = [1+(1+2\kappa_2)^\frac{1}{2}]\frac{\kappa^\frac{1}{2}_2}{(1+2\kappa_2)^\frac{1}{2}} \sim \kappa^\frac{1}{2}_2 \Rightarrow \kappa_2 \sim K.
\end{equation}
Therefore the increase in total material $K$ increases the network asymmetry $\frac{\kappa_2}{\kappa_1}$, as also suggested by numerical results (Fig.~\ref{fig:contraint_penalty}B). From this example we can see not only the constraint formulation might result in different network from that of penalty function formulation, but even when using the constraint formulation key qualitative features of optimal network may depend on the total material allocated to the network, a fact that has apparently not received scrutiny.

In this paper we will discuss the consequences of general boundary conditions on optimal networks, as well as the effect of different formulations of material cost. We will focus on networks minimizing transport costs, since these have recieved the most attention to date \cite{bohn2007structure,katifori2010damage,banavar2000topology}. We show that under the most general boundary conditions pathologies associated with minimizing dissipation are overcome if one instead minimizes a complementary energy that includes work done by pressure vertices. A network with minimal complementary energy is simply connected for all boundary conditions, a property previously only proven for minimally dissipative networks with Neumann boundary conditions. Networks optimizing complementary energy resolve pathological networks like the one in Fig.~\ref{fig:diss_loop} by disconnecting pressure vertices with the same pressure. The complementary energy reduces to dissipation when all the pressure vertices have the same pressure, so previous theoretical results for optimal networks are recovered. If at least one vertex with Neumann boundary condition is present minimally dissipative networks will disconnect all the Dirichlet vertices from each other, so ultimately our results provide a formal proof that minimally dissipative networks satisfy Murray's law, are simply connected, and disconnect pressure vertices under this narrower set of boundary conditions.

Throughout we model material costs via holonomic constraints, rather than using Murray's original approach of using penalty functions. The final leg of our argument is to elucidate the conditions under which the two formulations are equivalent; that is, they produce the same family of optimal networks as the cost or penalty parameters are varied. In particular we show that the two formulations are equivalent if the network flows are not affected by uniform rescaling of conductances, a property held by any network in which all pressure vertices have identical pressures, including any network that minimizes the complementary energy.

Taken together, our results comprehensively expose the effect of boundary conditions, especially vertices with specified pressures, and of formulations of material cost on minimally dissipative networks. It also suggests an energy function that incorporates the work done by pressure vertices that may be a more suitable target function for optimization than dissipation.

\section{Notation} \label{sec:notation}

\begin{figure}[t]
\begin{center}

\includegraphics[width = 7cm]{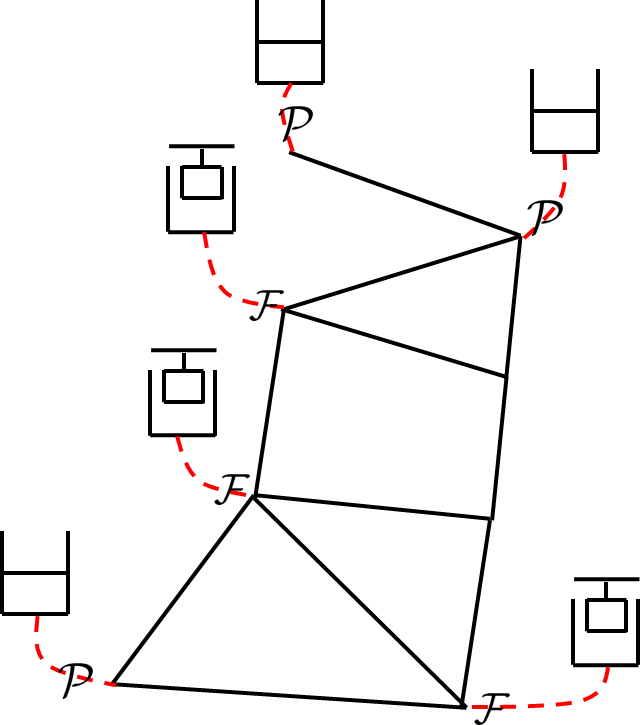}
\caption{A network diagram showing Dirichlet (pressure) vertices $\Pn$ and Neumann (flow) vertices $\F$, along with vertices where no boundary condition is imposed.}
\label{fig:network_diagram}

\end{center}
\end{figure}

In this work we consider a set of vertices $k = 1,...,V$ that connect to each other by vessels or edges. We indicate that vertices are neighbors in the network by writing $\langle k,l\rangle =1$ if vertices $k,l$ are linked by an edge and $\langle k,l\rangle = 0$ otherwise. This relation between vertices is symmetric, in the sense that $\langle k,l\rangle = \langle l,k\rangle$. If $\langle k,l\rangle =1$ a non-negative conductance $\kappa_{kl}$ and a flow $Q_{kl}$ are associated with the edge, with $\kappa_{kl} = \kappa_{lk}$ and $Q_{kl} = -Q_{lk}$. We model flows within hydraulic networks by assuming that a pressure $p_k$ can be assigned to each vertex $k$ and there is a linear relation between flow and pressure difference, i.e. $Q_{kl} = (p_k-p_l)\kappa_{kl}$. At every Neumann vertex the conservation of mass has to hold, i.e. $\sum_{l\,:\, \langle k,l\rangle =1} Q_{kl} = q_k$, where $q_k$ is the flow into the network at vertex $k$. We divide the vertices of the network into two classes: Neumann vertices where flow into the network is known, and Dirichlet vertices at which pressure is prescribed. Vertices that are not connected to external fluid sources, sinks or reservoirs are typically of Neumann type, with inflow $q_k = 0$ (Fig.~\ref{fig:network_diagram}). Let $\Pn$ denote the set of pressure (ore Dirichlet) vertices and $\F$ the set of flow (or Neumann) vertices (we require $\Pn\cap \F = \phi$). For definiteness we say $k\notin \Pn\cup \F$ if no boundary condition is imposed, i.e. $\sum_{l\,:\, \langle k,l\rangle =1} Q_{kl} = 0$. A Kirchhoff flow is defined as the flow $Q_{kl} = (p_k-p_l)\kappa_{kl} \; \forall \langle k,l\rangle =1$, where the pressures satisfy

\begin{equation}
\left\{\begin{array}{lll}
\sum_{l\,:\, \langle k,l\rangle =1} (p_k-p_l)\kappa_{kl} = 0 & k\notin \Pn\cup \F\\
p_k = \bar{p_k} & k\in \Pn\\
\sum_{l\,:\, \langle k,l\rangle =1} (p_k-p_l)\kappa_{kl} - q_k = 0 & k\in \F
\end{array}\right..
\end{equation}

\noindent It is well-known that for connected networks, i.e. $\forall 1\leq k,l\leq V$ we can devise a path from $k$ to $l$; that is: $\exists k_1,k_2,...,k_n$ s.t. $\langle k,k_1\rangle = \langle k_i,k_{i+1}\rangle= \langle k_n,l\rangle = 1 \; \forall 1\leq i \leq n-1$ and $\kappa_{k k_1},\kappa_{k_i k_{i+1}},\kappa_{k_n l}>0 \; \forall 1\leq i \leq n-1$, and $\Pn \neq \phi$, then the pressures are uniquely determined and therefore is the Kirchhoff flow \cite{LP:book}. In case $\Pn = \phi$ the Kirchhoff flow is uniquely determined so long as $\sum_{k\in \F} q_k = 0$, and pressures are determined up to an additive constant. If the condition on total inflow is violated there is no solution for $p_k$'s and the pressures are ill-defined. This result is quite important for developing intuition about the role of Dirichlet vertices in networks so we give a proof in the Appendix. On the otherhand if the network is not connected it consists of finitely many connected components, and each component would have to satisfy the condition for the Kirchhoff flow to be uniquely determined. We define a {\it physical network} to be a network with whose conductances $\kappa_{kl}$ and pressure conditions admitting a unique Kirchhoff flow solution.\\

\section{Results} \label{sec:results}

In this section we state the main results of this paper, identifying several properties of physical networks that globally minimize the dissipation and the complementary dissipation defined below:

\begin{definition}
The dissipation function given flows $Q_{kl}$ and conductances $\kappa_{kl}$ for $\langle k,l\rangle =1$ is defined by

\begin{equation}
D = \sum_{k>l,\langle k,l\rangle =1} \frac{Q_{kl}^2}{\kappa_{kl}}
\label{func:diss}
\end{equation}

\end{definition}

\begin{definition}
The complementary dissipation of a network given flows $Q_{kl}$ and conductances $\kappa_{kl}$ for $\langle k,l\rangle =1$ is defined by

\begin{equation}
f = \sum_{k>l,\langle k,l\rangle =1} \frac{Q_{kl}^2}{\kappa_{kl}} - 2\sum_{k\in \Pn} p_k \sum_{l\, :\, \langle k,l\rangle =1} Q_{kl}
\label{func:tot_energy}
\end{equation}

\end{definition}
We call $f$ the complementary dissipation because it resembles the complimentary energy, which in linear elasticity allows the displacement field to be calculated via minimization of a function. Notably this function, the complementary energy, is defined to be equal to the stored internal elastic energy minus the work done by any external traction, which is similar to our expression (rate of dissipation minus twice the rate of working of external tractions). We introduce the material constraint as

\begin{equation}
K = \sum_{k>l,\langle k,l\rangle =1} \kappa^\frac{1}{2}_{kl} d_{kl}.
\label{eq:material_constraint_d}
\end{equation}

\noindent Here $d_{kl} = \ell_{kl}^\frac{3}{2}$, where $\ell_{kl}$ is the length of link $kl$ in the hydraulic network. The $d_{kl}$ can be any set of positive weights for generality. A fundamental question is whether a global minimizer of dissipation (\ref{func:diss}) or complementary dissipation (\ref{func:tot_energy}) exists under material constraint or penalty:

\begin{proposition} \label{prop:existence_global_minimizer}
Suppose the network topology and boundary conditions are physical, i.e. $\kappa_{kl} >0 \; \forall \langle k,l\rangle = 1$ results in a physical network. Then there exists a physical network that globally minimizes dissipation (\ref{func:diss}) or complementary dissipation (\ref{func:tot_energy}) under material constraint (\ref{eq:material_constraint_d}). In addition there exists a physical network that minimizes dissipation (\ref{func:diss}) under material penalty, i.e. minimizes

\begin{equation}
\Theta = \sum_{k>l,\langle k,l\rangle =1} \frac{Q_{kl}^2}{\kappa_{kl}} + a\sum_{k>l,\langle k,l\rangle =1} \kappa^\frac{1}{2}_{kl} d_{kl}
\label{eq:constraint_material_penalty}
\end{equation}
under no constraint.

\end{proposition}
Observe that the complementary dissipation (\ref{func:tot_energy}) with material penalty might not have a global minimizer. Consider a simple network made up of two pressure vertices with prescribed pressures $p=1, 0$ connected by an edge with conductance $\kappa$. Then the complementary dissipation with material penalty is $-\kappa + a \kappa^\frac{1}{2}d_{12}$, which goes to $-\infty$ as $\kappa \to \infty$. Thus a global minimizer does not exist in this example. Now we define Murray's law:

\begin{definition}

A physical network is said to satisfy Murray's law if there is a constant $a>0$ such that the following relation between Kirchhoff flow $Q_{kl}$ and conductance $\kappa_{kl}$ holds $\forall \langle k,l\rangle =1$:

\begin{equation}
\kappa_{kl} = a \frac{|Q_{kl}|^\frac{4}{3}}{d^\frac{2}{3}_{kl}}.
\label{eq:murray_law}
\end{equation}

\end{definition}

\noindent If flows obey the Hagen-Poiseuille law (so that $\kappa_{kl} \; \propto \; r_{kl}^4$ where $r_{kl}$ is the radius of edge $kl$), then Equation \ref{eq:murray_law}) implies that $|Q_{kl}| \; \propto \; r^3_{kl}$. Our first result reframes Murray's law with respect to global minimizers.

\begin{theorem}\label{thm:murray_law}
A physical network that globally minimizes the complementary dissipation (\ref{func:tot_energy}) under material constraint (\ref{eq:material_constraint_d}) satisfies the Murray's law.
\end{theorem}

\noindent  Our second and third results establish properties previously attributed to minimal dissipative networks \cite{durand2007structure} but now allowing for both Neumann and Dirichlet boundary conditions. Let a path between vertices $k,l$ be a set of vertices $k=k_1,k_2,...,k_n = l$ such that no vertex is listed more than once and $\langle k_i,k_{i+1}\rangle = 1, \kappa_{k_i k_{i+1}} >0 \; \forall 1\leq i \leq n-1$.

\begin{theorem}\label{thm:no_loop}

In a physical network that globally minimizes the complementary dissipation (\ref{func:tot_energy}) under material constraint (\ref{eq:material_constraint_d}) there is exactly one path between any pair of points, except the case this network has no flow in it, i.e. $Q_{kl} = 0 \; \forall \langle k,l\rangle =1$.

\end{theorem}

\begin{theorem}\label{thm:no_pressure_path}

A physical network that globally minimizes the complementary dissipation (\ref{func:tot_energy}) under material constraint (\ref{eq:material_constraint_d}) has no path connecting 2 pressure vertices with the same prescribed pressure, except the case that the network has no flow in it.

\end{theorem}

\noindent From these results we can rederive properties of minimal dissipative networks for boundary conditions considered by Durand \cite{durand2007structure}.

\begin{corollary} \label{cor:energy_diss_equi}

A physical network that globally minimizes dissipation (\ref{func:diss}) under the material constraint (\ref{eq:material_constraint_d}) satisfies Murray's law, has no loops, and has no paths connecting two pressure vertices if all the pressure vertices have the same specified pressure
\end{corollary}

\begin{proof}

It suffices to show that the complementary dissipation (\ref{func:tot_energy}) reduces to dissipation (\ref{func:diss}). Since Kirchhoff flow remains the same up to an additive constant on all pressures we can without loss of generality let $p_0 = 0$. Then $f = D$ and the results carry through.
\end{proof} 

\noindent While it is possible that two pressure vertices with different prescribed pressures connect in networks with minimal complementary dissipation, it does not happen for minimal dissipative networks that have at least one vertex with flow boundary condition.

\begin{proposition} \label{prop:diss_no_pressure_path}

In a physical network that globally minimizes the dissipation (\ref{func:diss}) under material constraint (\ref{eq:material_constraint_d}) with $\F \neq \phi$ no pair of pressure vertices are connected by a path.
\end{proposition}

\noindent This along with Corollary \ref{cor:energy_diss_equi} establishes a general result on minimally dissipative networks

\begin{corollary} \label{cor:diss_result}

A physical network that globally minimizes the dissipation (\ref{func:diss}) under material constraint (\ref{eq:material_constraint_d}) with $\F \neq \phi$ satisfies Murray's law (\ref{eq:murray_law}), has no loops in the sense of Theorem \ref{thm:no_loop}, and has no paths connecting two pressure vertices in the sense of Proposition \ref{prop:diss_no_pressure_path}.

\end{corollary}

\begin{proof}

Suppose we have a physical network that globally minimizes dissipation (\ref{func:diss}) with $|\Pn| = n$ and the connected components (where two vertices can be connected only by links with positive conductance) of the network are labeled $G_1,G_2,...,G_m$. From Proposition \ref{prop:diss_no_pressure_path} we know that two pressure vertices cannot connect, so $m\geq n$ and each subgraph includes at most one pressure vertex, i.e. $|G_i\cap \Pn| \leq 1 \; \forall 1\leq i \leq m$. Now we look at a specific subnetwork $G_i$. The subnetwork satisfies the assumptions of Corollary \ref{cor:energy_diss_equi}, so it has to satisfy Murray's law (\ref{eq:murray_law}) and also contain no loops; or else there is no flow in $G_i$. Since this argument holds for all subnetworks the whole network satisfies Murray's law and contains no loops.
\end{proof}

Throughout this work we follow recent work \cite{bohn2007structure,katifori2010damage} by imposing material cost as a constraint rather than following Murray's approach of imposing it as a penalty function. Here we discuss the conditions under which these different formulations are equivalent for minimally dissipative networks.

\begin{proposition}\label{prop:material_formulation_scale_invar}
Suppose the flows in each minimally dissipative network under material constraint (\ref{eq:material_constraint_d}) are invariant when conductances are uniformly rescaled, i.e. the network with $\kappa'_{kl} = \beta \kappa_{kl}, \beta >0$ has the same flows as that in the original network. Then there is a bijection $K(a)$ from $(0,\infty)$ to $(0,\infty)$ such that every minimally dissipative network with material constraint $K$ is a minimally dissipative network with material penalty under some coefficient $a$ (\ref{eq:constraint_material_penalty}) and vice versa.
\end{proposition}
For networks with at least one flow boundary condition we know from Cor.~\ref{cor:diss_result} that all the pressure vertices disconnect and hence the flows in minimally dissipative networks under material constraint are invariant when conductances are uniformly rescaled. Thus

\begin{corollary}\label{cor:diss_flowBC_material_equi}
If the network has at least one flow vertex, i.e. $\F\neq \phi$, then the minimal dissipation problem under material constraint (\ref{eq:material_constraint_d}) and material penalty (\ref{eq:constraint_material_penalty}) are equivalent in the sense of Proposition \ref{prop:material_formulation_scale_invar}.
\end{corollary}

\section{Proof of Proposition \ref{prop:existence_global_minimizer}}\label{sec:existence_proof}

\begin{proof}
To begin consider the dissipation function (\ref{func:diss}) under material constraint (\ref{eq:material_constraint_d}). Suppose there are $E$ edges then the intersection of $\{\kappa_i \geq 0\}$ and the material constraint surfaces forms a compact set $A$ in $\R^E$. For each physical network the flow is obtained by inverting an invertible matrix with components continuously dependent on the conductances so the dissipation is continuous in the conductances. Dissipation is finite at each physical network since $\kappa_{kl}=0 \Rightarrow Q_{kl} = 0$. However not all the networks in this set are physical, specifically when a subnetwork with unbalanced flow boundary conditions is separated out, and we need to exclude non-physical networks but keep the set compact. From assumption $\kappa_{kl} > 0 \; \forall \langle k,l\rangle =1$ results in a physical network, so a non-physical network must have a set of edges $k_i l_i$ with $\kappa_{k_i l_i} = 0$ for $i=1,2,...,n$. It suffices to show that physical networks with $\kappa_{k_1 l_1},...,\kappa_{k_n l_n} < \epsilon$ have dissipation uniformly converging to infinity as $\epsilon \to 0^+$, so we can exclude this set without excluding a possible global minimum. If $\kappa_{k_1 l_1},...,\kappa_{k_n l_n} = 0$ gives a non-physical network there will be a connected component $C$ connected by $kl \notin \{k_1 l_1,...,k_n l_n\}$ and with $\sum_{i\in C\cap \F} q_i \neq 0$ but $\Pn\cap C =\phi$. Without loss of generality let $q_{tot}=\sum_{i\in C\cap \F} q_i > 0$ and assume $k_1 l_1,...,k_m l_m$ with $m\leq n$ connect $C$ with the rest of the network, i.e. $k_i \in C$ and $l_i \notin C \; \forall 1\leq i \leq m$. Then the unbalanced flow in $C$ must flow out through $k_1 l_1,...,k_m l_m$ so

\begin{equation}
\sum_{i=1}^m Q_{k_i l_i} = q_{tot} \Rightarrow \exists 1\leq j \leq m \; s.t. \; Q_{k_j l_j} \geq \frac{q_{tot}}{m}.
\end{equation}
Then

\begin{equation}
D = \sum_{k<l,\langle k,l\rangle =1} \frac{Q_{kl}^2}{\kappa_{kl}} \geq \frac{Q_{k_j l_j}^2}{\kappa_{k_j l_j}} \geq \frac{q^2_{tot}}{m^2 \epsilon} \geq \frac{q^2_{tot}}{n^2\epsilon}.
\end{equation}
Since $q_{tot}$ is independent of $\epsilon>0$ the dissipation of physical networks in the set $\{\kappa_{k_i l_i} <\epsilon | 1\leq i \leq n\}$ goes to infinity uniformly as $\epsilon \to 0^+$. Now for each non-physical network we can identify all the edges with zero conductance and create this set, with $\epsilon>0$ chosen such that $\epsilon < \frac{K^2}{(\sum_{k>l, \langle k,l\rangle =1} d_{kl})^2}$, where $K$ is the prescribed material cost, and all the physical networks within this set have dissipation greater than that of the uniform conductance network, i.e. $\kappa_i = \frac{K^2}{(\sum_{k>l, \langle k,l\rangle =1} d_{kl})^2} \; \forall 1\leq i \leq E$. Then if we exclude this set of networks from $A$ we will obtain a non-empty set (since the uniform conductance network is in the set) and we will not exclude the global minimum (since the uniform conductance network has lower dissipation than all the physical networks in the excluded set). Now we repeat this procedure for all $k_1 l_1,...,k_n l_n$ if zero conductance on these edges produces a non-physical network. Since there are only finitely many of them and each operation produces a compact set we know the remaining set is still compact. Then a globally minimally dissipative network exists since continuous function always achieves its global minimum on compact sets. The proof for dissipation with material penalty (\ref{eq:constraint_material_penalty}) follows along the same lines except that now $A$ is defined by $ \{\sum_{k<l, \langle k,l\rangle =1}\kappa^\frac{1}{2}_{kl} d_{kl} \leq M|\kappa_i \geq 0\}$ and $M$ chosen to be larger than the dissipation with material penalty (\ref{eq:constraint_material_penalty}) of the uniform conductance network.

Finally we consider the complementary dissipation (\ref{func:tot_energy}) with material constraint (\ref{eq:material_constraint_d}). The proof is similar except that now we need to establish a uniform upper bound of $|\sum_{k\in \Pn}p_k \sum_{l \;:\;\langle k,l\rangle = 1} Q_{kl}|$ for all physical networks in $A$. Then since the pressure work term is continuous with the conductance and we can exclude non-physical networks once we have this bound we can prove the existence of global minimizer as above. Since the flows depend linearly on the boundary conditions we can write $Q_{kl} = Q\ub{f}_{kl} + Q\ub{p}_{kl}$, where $\{Q\ub{f}_{kl}\}$ is obtained by setting all pressure vertices to have zero pressure and keeping all the flow boundary conditions and $\{Q\ub{p}_{kl}\}$ by setting all the flow vertices to have zero flow (i.e. remove the flow boundary condition on all flow vertices) and keeping all the pressure boundary conditions. It suffices to bound the pressure work term in these flows separately. In the network with $Q\ub{p}_{kl}$ notice that the maximum principle applies, i.e. if we let $\bar{p} = \max_{p_i, i\in \Pn}, \ul{p} =\min_{p_i, i\in \Pn}$ we have

\begin{equation}
\ul{p}\leq p_i \leq \bar{p}
\end{equation} 
for all vertices, $i$, that are connected to a pressure vertex (let the set of $i\notin \Pn$ and $i$ connected to a pressure vertex be $C$). This is obvious if $i\in \Pn$. If $i\in C$ Kirchhoff's first law at vertex $i$ may be rewritten as:

\begin{equation}
p_i = \frac{\sum_{j\;:\;\langle i,j\rangle =1} p_j\kappa_{ij}}{\sum_{j\;:\;\langle i,j\rangle =1} \kappa_{ij}}
\label{eq:p_mean_value_principle}
\end{equation}
($\sum_{j\;:\;\langle i,j\rangle =1} \kappa_{ij} > 0$ since $i$ connects to a pressure vertex and hence must connect to at least one adjacent vertex). Suppose for contradiction that $\exists p_{i_0} < \ul{p}$ with $i_0\in C$. Then we can without loss of generality have $p_{i_0} \leq p_j \; \forall j\in C$, and for Equation (\ref{eq:p_mean_value_principle}) to hold we must have $p_j = p_{i_0} \; \forall \kappa_{j i_0} >0$. By assumption $i_0$ connects to a pressure vertex $k\in \Pn$ so $p_k = p_{i_0} < \ul{p}$, a contradiction. Similarly we can prove that $p_i\leq \bar{p}, \; \forall i\in C$. Thus if we let the maximum degree of all the vertices be $d$ we have

\begin{equation}
|\sum_{k\in \Pn} p_k \sum_{l\, :\, \langle k,l\rangle =1} Q_{kl}\ub{p}| \leq |\Pn|\max\{|\bar{p}|,|\ul{p}|\} (\bar{p}-\ul{p})d \frac{K^2}{\min\{d_{kl}\}^2}
\end{equation}
which is a uniform bound for all the physical networks satisfying the material constraint (\ref{eq:material_constraint_d}). Now we consider the pressure work term with $Q\ub{f}_{kl}$. Without loss of generality we can assume $|\F| = 1$ since we can split the flow boundary condition to $\{Q\ub{f,1}_{kl}\},...,\{Q\ub{f,|\F|}_{kl}\}$, where $\{Q\ub{f,i}_{kl}\}$ is the flow in which only the $i_{th}$ flow boundary condition is applied, and for concreteness we let $q_{i_f} < 0$ where $i_f$ denotes the only flow vertex in the network. Now we focus on a specific $\{Q\ub{f,i}_{kl}\}$ and abbreviate it as $\{Q_{kl}\}$. We claim that $0\leq \sum_{l\;:\;\langle k,l\rangle = 1} Q_{kl} \leq -q_{i_f} \; \forall k\in \Pn$. Suppose for contradiction that $ \sum_{l\;:\;\langle k_0,l\rangle = 1} Q_{k_0 l} < 0$ for a $k_0\in \Pn$. Then $\exists k_1$ such that $\langle k_1,k_0\rangle = 1$ and $p_{k_1}> p_{k_0}$ by the assumption. If $k_1 \in \Pn$ we have a contradiction since $p_{k_1} = p_{k_0} = 0$, so $k_1 \notin \Pn \cup \F$ or $k_1 \in \F$. In either case we have $\sum_{l\neq k_0\;:\;\langle k_1,l\rangle =1} Q_{k_1 l} = q_{k_1} - Q_{k_1 k_0}$, where $q_{k_1} = q_{i_f}$ if $k_1\in \F$ and is zero otherwise. Thus the left hand side sums to a non-positive number so we can find $k_2$ such that $\langle k_2,k_1\rangle =1$ and $p_{k_2}>p_{k_1}$. Following this procedure we can find distinct $k_0, k_1,...,k_n$ such that $p_{k_i}>p_{k_{i-1}} \; \forall 1\leq i \leq n$ (if any two of the vertices are the same they would have the same pressure). Since $n>0$ is arbitrary we can let $n =V$, the number of vertices, so one of $k_i$ must belong to $\Pn$, a contradiction. The statement $\sum_{l\;:\;\langle k,l\rangle = 1} Q_{kl} \leq -q_{i_f}$ comes from the fact

\begin{equation}
\sum_{k\in \Pn}\left( \sum_{l\;:\;\langle k,l\rangle = 1} Q_{kl} \right) = -q_{i_f}
\end{equation}
so if $\exists k_0 \in \Pn$ such that $\sum_{l\;:\;\langle k_0,l\rangle = 1} Q_{k_0 l} > -q_{i_f}$ there must be a $k'_0 \in \Pn$ such that $\sum_{l\;:\;\langle k'_0,l\rangle = 1} Q_{k'_0 l} <0$, a contradiction. Similar estimates for $q_{i_f} > 0$ can be obtained in the same manner. With the estimates on the inflow into pressure vertices we have

\begin{equation}
|\sum_{k\in \Pn} p_k \sum_{l\, :\, \langle k,l\rangle =1} Q_{kl}\ub{f}| \leq |\Pn|\max\{|\bar{p}|,|\ul{p}|\}|q_{i_f}|.
\end{equation}
With these estimates we established an upper bound for the pressure work term and hence the global minimizer for the complementary dissipation (\ref{func:tot_energy}) under material constraint (\ref{eq:material_constraint_d}).
\end{proof}

\section{Proof of Theorem \ref{thm:murray_law}} \label{sec:proof_murray}

\begin{proof}
Consider a physical network that does not satisfy Murray's law, and we will show that this is not a global minimizer of complementary dissipation (\ref{func:tot_energy}) under material constraint (\ref{eq:material_constraint_d}). Suppose our network has flows and conductances $\tilde{Q}_{kl}, \tilde{\kappa}_{kl}$, and assume for now that $\tilde{Q}_{kl} \neq 0 \; \forall \langle k,l\rangle =1$. Now define $\kappa_{kl}$ to be the conductances that satisfy Murray's law (\ref{eq:murray_law}) and the material constraint (\ref{eq:material_constraint_d}) based on the fluxes in our original network, i.e.

\begin{equation}
\kappa_{kl} = a\frac{|\tilde{Q}_{kl}|^\frac{4}{3}}{d^\frac{2}{3}_{kl}} \quad \forall \langle k,l\rangle =1,\quad K = \sum_{k>l,\langle k,l\rangle =1} \kappa^\frac{1}{2}_{kl} d_{kl}
\label{eq:apply_Murray_law}
\end{equation}

\noindent where $a>0$ is uniquely determined by the material constraint. We show that this comparative network has strictly smaller complementary dissipation (\ref{func:tot_energy}), i.e.

\begin{equation}
\sum_{k>l,\langle k,l\rangle =1} \frac{\tilde{Q}_{kl}^2}{\kappa_{kl}} - 2\sum_{k\in \Pn} p_k \sum_{l\, :\, \langle k,l\rangle =1} \tilde{Q}_{kl} < \sum_{k>l,\langle k,l\rangle =1} \frac{\tilde{Q}_{kl}^2}{\tilde{\kappa}_{kl}} - 2\sum_{k\in \Pn} p_k \sum_{l\, :\, \langle k,l\rangle =1} \tilde{Q}_{kl}.
\label{ineq:murray_energy_minimization}
\end{equation}
We show this inequality by proving that the conductances satisfying Murray's law is the global minimizer of complementary dissipation (\ref{func:tot_energy}) when flows $\tilde{Q}_{kl}$ are held constant and the material constraint (\ref{eq:material_constraint_d}) is imposed. Consider the dissipation with a Lagrange multiplier imposing material constraint (since the pressure work term does not change when $\tilde{Q}_{kl}$ are held fixed)

\begin{equation}
\Theta = \sum_{k>l,\langle k,l\rangle =1} \frac{\tilde{Q}_{kl}^2}{\kappa_{kl}} + \lambda (\sum_{k>l,\langle k,l\rangle =1} \kappa^\frac{1}{2}_{kl} d_{kl} - K).
\end{equation}

\noindent First we find the stationary points:

\begin{equation}
0= \frac{\partial\Theta}{\partial \kappa_{kl}} = - \frac{\tilde{Q}^2_{kl}}{\kappa^2_{kl}} + \frac{\lambda}{2} \kappa^{-\frac{1}{2}}_{kl} d_{kl} \Rightarrow \kappa_{kl} = 2^\frac{2}{3} \frac{|\tilde{Q}_{kl}|^\frac{4}{3}}{(\lambda d_{kl})^\frac{2}{3}} \quad \forall \langle k,l\rangle =1
\label{eq:murraykappa}
\end{equation}

\noindent which is Murray's law when Hagen-Poiseuille's law is applied. Now $\lambda$ can be solved for by plugging Equation (\ref{eq:murraykappa}) back into the material constraint (\ref{eq:material_constraint_d}). Since the material constraint (\ref{eq:material_constraint_d}) along with $\kappa_{kl}\geq 0 \; \forall \langle k,l\rangle =1$ forms a compact set this is the unique global minimum so long as no minima occur on the boundaries, i.e. there is no local minimum for which $\exists \langle k,l\rangle =1$ s.t. $\kappa_{kl} = 0$. However since $\tilde{Q}_{kl} \neq 0 \; \forall\langle k,l\rangle =1$ any $\kappa_{kl} = 0$ will result in $f=\infty$ and thus global minimizers cannot happen on boundaries. Since the conductances that satisfy Murray's law on the material constraint surface is the only stationary point in the interior, and we have dispensed with global minima on the boundary it must be the unique global minimizer, and the inequality (\ref{ineq:murray_energy_minimization}) holds. Now to finalize our proof we remove the assumption that $\tilde{Q}_{kl}\neq 0 \; \forall \langle k,l\rangle =1$. Then we need to show that the new conductances $\kappa_{kl}$ along with original boundary conditions yield a physical network under the assumption that $\tilde{\kappa}_{kl}$ with boundary conditions gives a physical network, and that the conductances $\kappa_{kl}$ that satisfy Murray's law is still the unique global minimizer. The first aspect is trivial in the case $\tilde{Q}_{kl} \neq 0 \; \forall \langle k,l\rangle =1$ since this condition implies that $\tilde{\kappa}_{kl}, \kappa_{kl} >0 \; \forall \langle k,l\rangle =1$. However $\tilde{Q}_{kl} = 0$ does not imply $\tilde{\kappa}_{kl} = 0$ while $\kappa_{kl}$ will be zero, and the concern is that applying Equation (\ref{eq:apply_Murray_law}) will produce a set of disconnected networks that are not physical networks. Consider a connected subnetwork of $\{\tilde{\kappa}_{kl}\}$ containing some edges with zero flows $\tilde{Q}_{k_1 l_1} = \tilde{Q}_{k_2 l_2}=\cdots = \tilde{Q}_{k_n l_n} = 0$ (the statement that a network is a physical network is equivalent to all its connected subnetworks being physical networks). Assume for contradiction that a connected component of this subnetwork $G_s$ is not a physical network with conductances $\kappa_{kl}$. By the non-physical network assumption we have $G_s\cap \Pn = \phi$ and $\sum_{k\in G_s\cap \F} q_k \neq 0$. However since $\tilde{Q}_{k_i l_i} = 0 \; \forall 1\leq i \leq n$ we have $\sum_{k\in G_s} \sum_{l\,:\, \langle k,l\rangle =1, l\in G_s} \tilde{Q}_{kl} = \sum_{k\in G_s} \sum_{l\,:\, \langle k,l\rangle =1} \tilde{Q}_{kl} =  \sum_{k\in G_s\cap \F} q_k \neq 0$, contradicting the fact that there is a well-defined pressure $\tilde{p}_k$ on $G_s$ since $\sum_{k\in G_s} \sum_{l\,:\, \langle k,l\rangle =1, l\in G_s} \tilde{Q}_{kl} = \sum_{k\in G_s} \sum_{l\,:\, \langle k,l\rangle =1, l\in G_s} (\tilde{p}_k -\tilde{p}_l)\tilde{\kappa}_{kl} = \sum_{k,l\in G_s,k>l,\langle k,l\rangle =1} (\tilde{p}_k -\tilde{p}_l)\tilde{\kappa}_{kl} + (\tilde{p}_l -\tilde{p}_k)\tilde{\kappa}_{kl} = 0$.

Now we address the second aspect; namely that the set of conductances $\kappa_{kl}$ that satisfy Murray's law is still the unique global minimizer of dissipation under fixed flow $\tilde{Q}_{kl}$. Let us enumerate all the links with zero flows by $k_1 l_1,...,k_n l_n$. We have $n<E$ where $E$ is the number of edges since if all the flows are zero the network will already satisfy the Murray's law (\ref{eq:murray_law}) with $a = 0$. It suffices to show that any network with $\kappa_{k_i l_i} >0$ for some $i\in \{1,...,n\}$ cannot be a global minimizer. Then we can restrict ourselves on the surface $\kappa_{k_1 l_1} = \cdots = \kappa_{k_n l_n} = 0$ and do the same calculation (when $\tilde{Q}_{kl} = 0 \; \forall \langle k,l\rangle =1$ the network already satisfies the Murray's law with the constant $a=0$, so this case can be excluded). However the result is immediate in this case because if we set $\kappa_{k_i l_i} = 0 \; \forall i\in I$ and scale the rest of the conductances up by a multiplicative constant we will strictly reduce the dissipation, so it cannot be a global minimizer.

Now we fix the conductances $\kappa_{kl}$ and change the flows in order to satisfy Kirchhoff's laws. We claim that among all the flows that satisfy conservation of mass and flow boundary conditions, i.e. $\sum_{l,\langle k,l\rangle=1} Q_{kl} = 0$ if $k\notin \Pn\cup\F$ and $\sum_{l,\langle k,l\rangle=1} Q_{kl} - q_k = 0$ if $k\in \F$ the Kirchhoff flow minimizes the function (\ref{func:tot_energy}) with $\kappa_{kl}$ fixed. Then since the original flow $\tilde{Q}_{kl}$ lies in this category we can show

\begin{equation}
\sum_{k>l,\langle k,l\rangle =1} \frac{Q_{kl}^2}{\kappa_{kl}} - 2\sum_{k\in \Pn} p_k \sum_{l\, :\, \langle k,l\rangle =1} Q_{kl} \leq \sum_{k>l,\langle k,l\rangle =1} \frac{\tilde{Q}_{kl}^2}{\kappa_{kl}} - 2\sum_{k\in \Pn} p_k \sum_{l\, :\, \langle k,l\rangle =1} \tilde{Q}_{kl}
\label{ineq:murray_flow_minimization}
\end{equation}

\noindent which finishes the proof. To see this we can impose the Lagrange multipliers for conservation of mass and flow boundary conditions on function (\ref{func:tot_energy}):

\begin{equation}
\Theta = \sum_{k>l,\langle k,l\rangle =1} \frac{Q_{kl}^2}{\kappa_{kl}} - 2\sum_{k\in \Pn} p_k \sum_{l,\langle k,l\rangle =1} Q_{kl} - \sum_{k\notin \Pn} \lambda_k \left(\sum_{l,\langle k,l\rangle =1} Q_{kl} - q_k\right)
\end{equation}

\noindent where $\lambda_k$ are Lagrange multipliers (for convenience we set $q_k = 0$ if $k\notin \Pn\cup \F$). To minimize this function we take derivatives and set them to zero:

\begin{equation}
0 = \frac{\partial\Theta}{\partial Q_{kl}} = \frac{2Q_{kl}}{\kappa_{kl}} - (\lambda_k - \lambda_l)
\label{eq:flow_der}
\end{equation}

\noindent and we define $\lambda_k = 2p_k$ if $k\in \Pn$. If we apply conservation of flux and flow boundary condition on $k\notin \Pn$ in terms of $\lambda_k$'s, i.e. substituting $Q_{kl}$'s by $\lambda_k$'s using Equation (\ref{eq:flow_der}), and impose $\lambda_k = 2p_k$ for $k\in \Pn$, then $\lambda_k$'s satisfy the exact same equations as the pressure under Kirchhoff's laws. We know from Section \ref{sec:notation} that if $\Pn \neq \phi$ then the pressure has a unique solution; otherwise the pressure is determined up to an additive constant, which has no effect on the flows. Therefore the flows $Q_{kl}$'s always have a unique solution. To show that Kirchhoff flow is a global minimum of the complementary dissipation (\ref{func:tot_energy}) notice that now the conservation of mass and flow boundary condition constraints might not give us a compact set, so there is no boundary. However $f$ has quadratic growth in flow through any link, so we can find $M>0$ s.t. $f>2b$ whenever $|Q_{kl}|>M$ for any $\langle k,l\rangle =1$, where $b$ is the value of the complementary dissipation $f$ for Kirchhoff flow. Then since $f$ has a global minimum in the compact set $|Q_{kl}|\leq M \; \langle k,l\rangle =1$ and it cannot be on the boundary it will have to be the Kirchhoff flow, which establishes that the Kirchhoff flow is the unique global minimizer of the complementary dissipation (\ref{func:tot_energy}) given fixed conductances $\kappa_{kl}$, which finishes the proof. 
\end{proof}

\section{Proof for Theorem \ref{thm:no_loop}} \label{sec:proof_no_loop}

\begin{proof}
Consider a physical network that contains a loop, $e$, with at least 3 points, i.e. $k_1,...,k_n$ with $\langle k_i, k_{i+1}\rangle =1, \kappa_{k_i k_{i+1}}>0 \; \forall 1\leq i \leq n$ (we set $k_{n+1} = k_1$) and $n\geq 3$, and let $C = \{(k_1,k_2),...,(k_{n-1},k_n), (k_n,k_1)\}$ be the set of ordered pairs denoting all the edges in the loop. Without loss of generality we can assume that the loop does not intersect itself, i.e. $|\{k_1,...,k_n\}| = n$; otherwise we can choose a non-selfintersecting subloop from it and proceed with the subloop. First we assume that $Q_{k_1 k_2},...,Q_{k_n k_1}$ are not all the same. We knew from Section \ref{sec:proof_murray} that adjusting conductances according to Murray's law under material constraint will decrease the dissipation without changing the pressure work term in the complementary dissipation function (\ref{func:tot_energy}) and that the resulting network will remain physical, so we can decrease the complementary dissipation by adjusting the conductances on the loop according to Murray's law with the material on the loop fixed. Therefore without loss of generality we can assume that $\exists a>0$ s.t. $\tilde{\kappa}_{k_i k_{i+1}} = a \frac{|\tilde{Q}_{k_i k_{i+1}}|^\frac{4}{3}}{d^\frac{2}{3}_{kl}} \; \forall 1\leq i \leq n$. Now we consider adding in a loop current $Q$, that is we add the same current $Q$ to each edge in the loop, and adjust the conductances by Murray's law under material constraint, i.e. set

\begin{equation}
Q_{kl} = \tilde{Q}_{kl} + Q\quad \textrm{and}\quad \kappa_{kl} = \mu \frac{Q^\frac{4}{3}_{kl}}{d_{kl}^\frac{2}{3}} \qquad \forall (k,l) \in C
\label{eq:apply_loop_current}
\end{equation}

\noindent where

\begin{equation}
\mu = \frac{K\subsc{loop}^2}{(\sum_{(k,l)\in C} Q^\frac{2}{3}_{kl} d^\frac{2}{3}_{kl})^2}, \quad K_{\mbox{\scriptsize loop}}\doteq \sum_{(k,l)\in C} \tilde{\kappa}_{kl}^\frac{1}{2}d_{kl}
\end{equation}

\noindent (we say $(k,l) \in C$ if the ordered pair $(k,l) = (k_i k_{i+1})$ for some $1\leq i\leq n$). Notice that for any $Q\in \R$ the new flows $Q_{kl}, (k,l)\in C$ along with the original flows outside of the loop $\tilde{Q}_{kl}, \langle k,l\rangle = 1, (k,l), (l,k)\notin C$ still satisfy conservation of mass and flow boundary conditions since the addition of $Q$ does not change the total flow into any of the vertices. If $\{k_1,...,k_n\} \cap \Pn = 0$ then changing the flow will only change the dissipation on the loop, and we only need to consider

\begin{equation}
D_{\mbox{\scriptsize loop}} \doteq \sum_{(k,l)\in C} \frac{Q_{kl}^2}{\kappa_{kl}}.
\label{func:loop_diss}
\end{equation}

\noindent Suppose that our contains a certain number of pressure vertices: $k_{n_1},...,k_{n_m}\in \Pn$ with $m\leq n$. For any $k_{n_j}$ if we restrict the sum $\sum_{l\,:\,\langle k_{n_j},l\rangle} Q_{k_{n_j}l}$ to edges in the loop, then it can be written as $Q_{k_{n_j} ,k_{n_j}+1} + Q_{k_{n_j} ,k_{n_j}-1}$ (recall that we assumed the loop has no self-interception). Since $Q_{kl} = \tilde{Q}_{kl} + Q$ we will have $Q_{k_{n_j} ,k_{n_j}+1} + Q_{k_{n_j} ,k_{n_j}-1} = \tilde{Q}_{k_{n_j} ,k_{n_j}+1} + \tilde{Q}_{k_{n_j} ,k_{n_j}-1} \; \forall Q\in\R$ and the pressure work term does not change. Thus in either case if we find flows and conductances on the loop that decrease the dissipation on the loop (\ref{func:loop_diss}) they will decrease the complementary dissipation (\ref{func:tot_energy}) as well. Therefore if we show that $D_{\mbox{\scriptsize loop}}$ strictly decreases after adding a loop current (\ref{eq:apply_loop_current}), then the Kirchhoff flow on the new network will have lower complementary dissipation by the argument in Section \ref{sec:proof_murray}, a contradiction. Calculate

\begin{equation}
D_{\mbox{\scriptsize loop}} = \sum_{(k,l)\in C} \frac{Q_{kl}^2}{\kappa_{kl}} = \sum_{(k,l)\in C}  \frac{Q_{kl}^\frac{2}{3} d_{kl}^\frac{2}{3}}{\mu} = \frac{(\sum_{(k,l)\in C} Q^\frac{2}{3}_{kl} d^\frac{2}{3}_{kl})^3}{K_{\mbox{\scriptsize loop}}^2}.
\label{func:diss_loopcurrent}
\end{equation}

\noindent The derivative with respect to $Q$ is (we let $A = \sum_{(k,l)\in C} Q^\frac{2}{3}_{kl} d^\frac{2}{3}_{kl}$ for simplicity on notations)

\begin{equation}
\frac{d D_{\mbox{\scriptsize loop}}}{d Q} = \frac{2A^2}{K\subsc{loop}^2} \sum_{(k,l)\in C} Q^{-\frac{1}{3}}_{kl} d^\frac{2}{3}_{kl}.
\label{eq:loop_current_der}
\end{equation}

\noindent Since $Q_{kl}$ are not all the same for $(k,l)\in C$ we have $A> 0$ (and $K\subsc{loop}>0$ by definition) and the factor $\frac{2A^2}{K\subsc{loop}^2}$ is always positive, so the sign of derivative depends only on $\sum_{(k,l)\in C} Q^{-\frac{1}{3}}_{kl} d^\frac{2}{3}_{kl}$ in this case (we will discuss the case $A=0$ later). Now we show that $Q_{kl} = 0$ for some $(k,l)\in C$ is always a local minimum. Suppose $Q_{kl} = \epsilon$ where $\epsilon \to 0^+$. Then $Q^{-\frac{1}{3}}_{kl} \to \infty$ and we will have $\frac{d D_{\mbox{\scriptsize loop}}}{d Q}  > 0$. The same argument applies to $Q_{kl} = -\epsilon$ so $Q_{kl} = 0$ is indeed a local minimum. To show that global minima can only happen when $Q_{kl} = 0$ for some $(k,l)\in C$ notice that there exists at least one global minimum since $D_{\mbox{\scriptsize loop}} \to \infty$ as $Q\to \pm \infty$ and $D_{\mbox{\scriptsize loop}}$ is a continuous function of $Q$. This global minimum may be attained only where $\frac{d D_{\mbox{\scriptsize loop}}}{d Q}  = 0$ or if the derivative is not defined. For the derivative to be not defined we will have at least one $Q_{kl} = 0$, which corresponds to a local minimum with a cusp in $D_{\mbox{\scriptsize loop}}$ as discussed. Now suppose $\frac{d D_{\mbox{\scriptsize loop}}}{d Q}  = 0$ so $B\doteq \sum_{(k,l)\in C} Q^{-\frac{1}{3}}_{kl} d_{kl}^\frac{2}{3} = 0$ and $Q_{kl}\neq 0 \; \forall (k,l)\in C$. Then we can take the second derivative:

\begin{equation}
\frac{d^2 D_{\mbox{\scriptsize loop}}}{d Q^2} = \frac{8AB^2}{3K\subsc{loop}^2} - \frac{2A^2}{3K\subsc{loop}^2}\sum_{(k,l)\in C} Q^{-\frac{4}{3}}_{kl} d^\frac{2}{3}_{kl} <0
\end{equation}

\noindent since by assumption $Q_{kl} \neq 0 \; \forall (k,l)\in C$. Thus the local extrema with $Q_{kl}\neq 0 \; \forall (k,l)\in C$ are all local maxima, and a global minimum will happen only if $\exists (k,l)\in C$ s.t. $Q_{kl} = 0$. Now we fix the conductances $\kappa_{kl} \; \forall (k,l)\in C$ and the original conductances outside the loop $\tilde{\kappa}_{kl} \; \forall (k,l),(l,k)\notin C$ and change the flow to Kirchhoff flow. If this is a physical network then as we have seen in Section \ref{sec:proof_murray} this process strictly decreases the complementary dissipation if the flow is not already the Kirchhoff flow, and the proof finishes since the step of adding a loop current $Q$ strictly decreases the dissipation on the loop and thus the complementary dissipation since a loop cannot be a global minimizer. It remains to show that the resulting network is a physical network. Suppose for contradiction that after adding a loop current $Q$ we can produce a non-physical connected subnetwork $G_s$ of $\{\kappa_{kl}\}$ by deleting the zero flux edges (when $(k,l)\notin C$ let $\kappa_{kl} = \tilde{\kappa}_{kl}$ be the original conductance since the procedure (\ref{eq:apply_loop_current}) does not change conductances outside of the loop). Similar to the proof in Section \ref{sec:proof_murray} it suffices to show that the original flow $\{\tilde{Q}_{kl}\}$ satisfies $\sum_{k\in G_s} \sum_{l\; : \; \langle k,l\rangle = 1} \tilde{Q}_{kl} = \sum_{k\in G_s} \sum_{l\; : \; \langle k,l\rangle =1, l\in G_s} \tilde{Q}_{kl}$ since the non-physical network assumption implies $\sum_{k\in G_s} \sum_{l\; : \; \langle k,l\rangle =1} \tilde{Q}_{kl} = \sum_{k\in \F\cap G_s} q_k \neq 0$, contradicting that $\sum_{k\in G_s} \sum_{l\; : \; \langle k,l\rangle = 1, l\in G_s} \tilde{Q}_{kl} = 0$ . To establish the equality we split the sum into the parts $k\in G_s \bs e$ and $k\in G_s\cap e$ where $e=\{k_1,...,k_n\}$ is the set of vertices in the loop. The equality $\sum_{k\in G_s \bs e} \sum_{l\; : \; \langle k,l\rangle = 1} \tilde{Q}_{kl} = \sum_{k\in G_s \bs e} \sum_{l\; : \; \langle k,l\rangle =1, l\in G_s} \tilde{Q}_{kl}$ holds because for $k\in G_s$ any edge connecting it does not lie in $C$, so $\langle k,l\rangle =1, l\notin G_s$ implies $0 = \kappa_{kl} = \tilde{\kappa}_{kl}$ and $\tilde{Q}_{kl} = 0$. When $k\in G_s\cap e$ we will have to consider connected components of $G_s\cap e$ of $\{\kappa_{kl}\}$ restricted in the loop $C$. Let $G_{1},...,G_{m}, m\leq n$ be those connected components, i.e. if $k\in G_{i}, l\in G_{j}, i\neq j$ then there is no path $k=l_1,...,l_h = l$ with $(l_i,l_{i+1})$ or $(l_{i+1},l_i)\in C, \kappa_{l_i l_{i+1}}>0 \; \forall 1\leq i \leq h-1$. Now consider a specific $G_{i}$ and let $k\ub{i}_1, k\ub{i}_2$ be its two end vertices (the only two vertices that are connected to only one vertex in $G_{i}$ by edges in $C$), with $l\ub{i}_1, l\ub{i}_2$ be the neighboring vertices in the loop that are not in $G_{i}$, i.e. $(k\ub{i}_1,l\ub{i}_1), (l\ub{i}_2,k\ub{i}_2)\in C, l\ub{i}_j \notin G_i, j=1,2$ (the order switching comes from the orientation of the edges). Then $\sum_{k\in G_{i}} \sum_{l\: : \; \langle k,l\rangle =1} \tilde{Q}_{kl} = \sum_{k\in G_{i}} \sum_{l\: : \; \langle k,l\rangle =1, l\in G_s} \tilde{Q}_{kl} + \sum_{j=1,2} \tilde{Q}_{k\ub{i}_j l\ub{i}_j}$ since again we do not have to consider flows on edges that are not in the loop. Now the sum $\sum_{j=1,2} \tilde{Q}_{k\ub{i}_j l\ub{i}_j}$ because $\kappa_{k\ub{i}_j l\ub{i}_j} = 0, j=1,2$ indicates that $\tilde{Q}_{k\ub{i}_1 l\ub{i}_1} = - \tilde{Q}_{k\ub{i}_2 l\ub{i}_2}$ since this is the only circumstance that an addition of a loop current eliminates both edges (the minus sign again comes from the orientation of the edges). Therefore $\sum_{k\in G_s\cap e} \sum_{l\: : \; \langle k,l\rangle =1, l\in G_s} \tilde{Q}_{kl} = \sum_{i=1}^m \sum_{k\in G_i}\sum_{l\: : \; \langle k,l\rangle =1, l\in G_s} \tilde{Q}_{kl} =  \sum_{i=1}^m \sum_{k\in G_i}\sum_{l\: : \; \langle k,l\rangle =1} \tilde{Q}_{kl} = \sum_{k\in G_s\cap e} \sum_{l\: : \; \langle k,l\rangle =1} \tilde{Q}_{kl}$ and the non-physical network hypothesis leads to a contradiction.

Now assume $\tilde{Q}_{kl} = Q_0 \in \R \; \forall (k,l)\in C$. In this case we must have $Q_0 = 0$ since otherwise when $Q_0 >0$ we will have $p_{k_1} >p_{k_2}>\cdots > p_{k_m} > p_{k_1}$, a contradiction, and similarly for $Q_0<0$. By assumption the network has at least one edge that has flow in it and does not comprise the loop, i.e. there is an edge $kl$ such that $(k,l), (l,k)\notin C$ and $Q_{kl}\neq 0$. Since the loop carries no flow we can set $\kappa_{kl} = 0 \; \forall (k,l)\in C$ without changing the complementary dissipation. To show that adding these materials back to edges with flows in them strictly decreases the complementary dissipation we prove a generalized Rayleigh's principle that allows for Dirichlet boundary conditions.
\begin{lemma}[Rayleigh's Principle]
The complementary dissipation (\ref{func:tot_energy}) monotonically decreases with the conductance of each edge, i.e. if we let $\{\tilde{\kappa}_{kl}\}, \{\tilde{Q}_{kl}\}$ be the sets of conductances and flows that satisfy all boundary conditions and $\{\kappa_{kl}\}, \{Q_{kl}\}$ be another sets of conductances with $Q_{kl}$ being the Kirchhoff flows, and they are on the same network with the same boundary conditions, then

\begin{equation}
\kappa_{kl}\geq \tilde{\kappa}_{kl} \quad \forall \langle k,l\rangle =1 \Rightarrow \sum_{k>l, \langle k,l\rangle =1} \frac{Q^2_{kl}}{\kappa_{kl}} - 2\sum_{k\in \Pn} p_k \sum_{l\; :\; \langle k,l\rangle =1} Q_{kl} \leq \sum_{k>l, \langle k,l\rangle =1} \frac{\tilde{Q}^2_{kl}}{\tilde{\kappa}_{kl}} - 2\sum_{k\in \Pn} p_k \sum_{l\; :\; \langle k,l\rangle =1} \tilde{Q}_{kl}.
\end{equation}
Moreover, if $\kappa_{kl}>\tilde{\kappa}_{kl}$ on an edge with $\tilde{Q}_{kl}\neq 0$, then the inequality holds.
\end{lemma}
\begin{proof}
To show the inequality we change the conductances and flows in two steps (we can without loss of generality change $\tilde{Q}_{kl}$ to the Kirchhoff flows corresponding to $\tilde{\kappa}_{kl}$ since from Section \ref{sec:proof_murray} we know doing so reduces the complementary dissipation). First we change the set of conductances from $\{\tilde{\kappa}_{kl}\}$ to $\{\kappa_{kl}\}$ and show that the complementary dissipation with the non-Kirchhoff flows $\tilde{Q}_{kl}$ decreases. Then we relax the flows to Kirchhoff flows $Q_{kl}$, which we know decreases the complementary dissipation from Section \ref{sec:proof_murray}. In the first step we can ignore the pressure work term $\sum_{k\in \Pn} p_k \sum_{l\; :\; \langle k,l\rangle =1} \tilde{Q}_{kl}$ since the flows remain unchanged and the pressures are prescribed. Then the fact $\kappa_{kl}\geq \tilde{\kappa}_{kl}$ implies that $\sum_{k>l, \langle k,l\rangle =1} \frac{\tilde{Q}^2_{kl}}{\kappa_{kl}} \leq \sum_{k>l, \langle k,l\rangle =1} \frac{\tilde{Q}^2_{kl}}{\tilde{\kappa}_{kl}}$, which finishes the proof. The strict inequality comes from that $\frac{\tilde{Q}^2_{kl}}{\kappa_{kl}} < \frac{\tilde{Q}^2_{kl}}{\tilde{\kappa}_{kl}}$ if $\tilde{Q}_{kl} \neq 0$ and $\kappa_{kl} > \tilde{\kappa}_{kl}$.
\end{proof}

If we let $\{\tilde{\kappa}_{kl}\}$ to be the set of original conductances but with $\tilde{\kappa}_{kl} = 0 \; \forall (k,l)\in C$, and $\{\tilde{Q}_{kl}\}$ be the set of the original flows, since $\sum_{k>l, \langle k,l\rangle =1} \tilde{\kappa}_{kl}^\frac{1}{2}d_{kl} < K$ we can find a new set of conductances $\{\kappa_{kl}\}$ with $\kappa_{kl} \geq \tilde{\kappa}_{kl}, \sum_{k>l, \langle k,l\rangle =1} \kappa_{kl}^\frac{1}{2}d_{kl} = K, \kappa_{kl} = 0 \; \forall (k,l)\in C$, and $\exists \langle k,l\rangle =1$ such that $\kappa_{kl} > \tilde{\kappa}_{kl}, \tilde{Q}_{kl}\neq 0$, then by Rayleigh's principle the complementary dissipation of $\{\kappa_{kl}\}, \{Q_{kl}\}$ will be strictly less than that of $\{\tilde{\kappa}_{kl}\}, \{\tilde{Q}_{kl}\}$ and the proof follows.
\end{proof}

\section{Proof for Theorem \ref{thm:no_pressure_path}} \label{sec:proof_no_pressure_path}

\begin{proof}
To prove that there is no path connecting 2 pressure vertices with the same pressure suppose there is a path $k_1,...,k_n$ with $\langle k_i k_{i+1}\rangle=1, \kappa_{k_i k_{i+1}} > 0 \; \forall i=1,...,n-1$, $n\geq 2$, and $k_1,k_n \in \Pn$ with $p_{k_1} = p_{k_n}$. As in Section \ref{sec:proof_no_loop} we can redistribute the conductances in the path to satisfy Murray's law with the material in this path held constant without increasing the complementary dissipation (\ref{func:tot_energy}). Without loss of generality We may assume that this path does not self-intersect because we can otherwise extract a subpath that does not self-intersect. Since we will only adjust flow and conductances on the path we can again restrict our attention to contribution of this path to the complementary dissipation:

\begin{equation}
f\subsc{path} = \sum_{(k,l)\in C} \frac{Q^2_{kl}}{\kappa_{kl}} - 2\sum_{k\in C_p} p_k \sum_{l\,:\,\langle k,l\rangle = 1, l\in C_n} Q_{kl}.
\end{equation}

\noindent Here as before we let $C$ be the set of ordered pairs of edges in the path, i.e. $C=\{(k_1,k_2),...,(k_{n-1},k_n)\}$, and $C_p$ be the set of all the pressure vertices in the path, and $C_n = \{k_1,...,k_n\}$. Now we add in a path current $Q$ that resembles the loop current in Section \ref{sec:proof_no_loop}, i.e.

\begin{equation}
Q_{kl} = \tilde{Q}_{kl} + Q,\quad \kappa_{kl} = \mu \frac{Q^\frac{4}{3}_{kl}}{d_{kl}^\frac{2}{3}} \qquad \forall (k,l) \in C
\label{eq:path_current}
\end{equation}

\noindent where

\begin{equation}
\mu = \frac{K\subsc{path}^2}{(\sum_{(k,l)\in C} Q^\frac{2}{3}_{kl} d^\frac{2}{3}_{kl})^2}, \quad K_{\mbox{\scriptsize path}}\doteq \sum_{(k,l)\in C} \tilde{\kappa}_{kl}^\frac{1}{2}d_{kl}
\label{eq:path_current_material_constraint}
\end{equation}

\noindent and $\tilde{Q}_{kl}, \tilde{\kappa}_{kl}$ denote the original flow and conductance, which according to Theorem \ref{thm:murray_law}, are related via Murray's law. We can see that if $k\in C_p$ but $k\neq k_1, k_n$ then $\sum_{l\,:\,\langle k,l\rangle = 1, l\in C_n} Q_{kl}$ consists of 2 terms in which the path current $Q$ cancels, so adding path current does not affect the pressure work terms for these vertices. Similarly the original flows are constants in the pressure work term and can be ignored if we only wish to tease out the dependence of $f\subsc{path}$. Thus up to an additive constant:

\begin{equation}
f\subsc{path} = \sum_{(k,l)\in C} \frac{Q^2_{kl}}{\kappa_{kl}} - 2(p_{k_1} - p_{k_n})Q = \sum_{(k,l)\in C} \frac{Q^2_{kl}}{\kappa_{kl}} = D\subsc{path}
\end{equation}

\noindent where the sign comes from our convention that the path current flows out of $k_1$ but flows into $k_n$. Thus the complementary dissipation reduces to dissipation on the path in this case. Also notice that adding a path current will not affect the conservation of mass and flow boundary conditions since it adds no flow to $k_2,...,k_{n-1}$ and $k_1,k_n$ are pressure vertices and do not have prescribed inflow, so if the procedure (\ref{eq:path_current}) strictly reduces the dissipation on the loop $f\subsc{path}$ we can relax the flows to Kirchhoff flows without increasing the complementary dissipation, which leads to a contradiction. Since $D\subsc{path}$ has the same form as $D\subsc{loop}$ in Section \ref{sec:proof_no_loop} we can prove in the same way that if $Q_{kl}$ are not all the same for $(k,l)\in C$ the global minimum only happens when $Q_{kl} = 0$ for some $(k,l)\in C$, which leads to a contradiction, and if $Q_{kl} = Q_0 \; \forall (k,l)\in C$ we must have $Q_0 = 0$ or we will have $p_{k_1}\neq p_{k_n}$, a contradiction. In this case by assumption we have an edge $kl$ with $(k,l), (l,k)\notin C$ and $Q_{kl}\neq 0$. Then similarly to Section \ref{sec:proof_no_loop} we can remove the materials on the path $C$ and apply Rayleigh's principle to decrease the complementary dissipation, which finishes the argument. One last issue needed to be addressed is whether cutting this path will result in a non-physical network. As in Section \ref{sec:proof_no_loop} we can define connected segments in the path after the path is cut and suppose for contradiction that there is a subnetwork connected to multiple connected segments. If this segment contains $k_1$ or $k_n$ then there is at least one pressure vertex and this subnetwork is physical. Otherwise this subnetwork connects to only connected segments in the middle which have the same original flows into and out of them, and we have $0 = \sum_{k\in G_s} \sum_{l\; :\; \langle k,l\rangle =1, l\in G_s} \tilde{Q}_{kl} = \sum_{k\in G_s} \sum_{l\; :\; \langle k,l\rangle =1} \tilde{Q}_{kl} \neq 0$ where $G_s$ is the non-physical subnetwork after cutting the path, a contradiction.
\end{proof}

\section{Proof for Proposition \ref{prop:diss_no_pressure_path}} \label{sec:proof_diss_no_pressure_path}

\begin{proof}
Suppose we start with a physical network that globally minimizes the dissipation (\ref{func:diss}) under the material constraint (\ref{eq:material_constraint_d}) with $n\doteq |\Pn|\geq 2$ (otherwise there is nothing to prove). Since the number of paths connecting two different pressure vertices is finite we can assume that there is a finite number of paths linking pressure vertices. On any path we can decrease the dissipation restricted on the path by adding a path current and adjust the conductances according to Murray's law as in Section \ref{sec:proof_no_pressure_path} in the case that not all the flows (with sign determined by the path direction) are the same, and by simply reducing all the flows to zero if they all agree and eliminating the whole path while scaling up the rest of the network by a multiplicative constant to meet the material constraint (\ref{eq:material_constraint_d}), given that this path does not comprise all the network. This procedure strictly reduces the dissipation since in the case not all the flows on the path are the same we will cut a proper set of them as in Section \ref{sec:proof_no_pressure_path}, which strictly decreases the dissipation. In case where all the flows are the same on the path because $|\F|\neq \phi$ and a flow vertex cannot lie on this path (otherwise the flows will not all be the same) so there will always be an edge not in this path with nonzero flow. Thus we can eliminate each path one at a time, strictly decreasing the dissipation while still satisfying the conservation of mass and flow boundary conditions and also the network remaining physical so long as we are not taking out the last path, in which case we have to worry about this path comprising the whole network. Up to now we do not solve for the flows according to Kirchhoff's laws since this might not decrease dissipation (it is only guaranteed to decrease the complementary dissipation). Notice that while we might take out multiple paths at a time in case of several paths sharing common links, the number of paths will never increase since no new edge with positive conductance can be created in this process. If we never reach the situation where we need to take out the last path, which is possible because eliminating one path may also disconnect others, then we do not have to worry about the path we are taking out might comprise the whole network (since there are other distinct paths), and we reach a network with connected components $G_1,...,G_m$ with $m\geq n$ since each component can contain at most one pressure vertex. Then the complementary dissipation function becomes

\begin{equation}
f = \sum_{k>l,\langle k,l\rangle =1} \frac{Q^2_{kl}}{\kappa_{kl}} - 2\sum_{k=1}^n p_k \sum_{l\,:\,\langle k,l\rangle =1} Q_{kl}
\end{equation}

\noindent if we without loss of generality let $k= 1,...,n$ be the pressure vertices in $G_1,...,G_n$ respectively. Since $Q_{kl} = 0$ when $k\in G_i, l\in G_j$ when $i\neq j$ we can isolate the contribution of a component $G_i$ to the complementary dissipation, starting from:

\[
0 = \sum_{k,l\in G_i,\langle k,l\rangle =1} (p_k-p_l)\kappa_{kl} = \sum_{k\in G_i} \left( \sum_{l\in G_i\,:\,\langle k,l\rangle = 1} Q_{kl}\right)
\]

\begin{equation}
= \sum_{k\in G_i}\left( \sum_{l\,:\,\langle k,l\rangle = 1} Q_{kl}\right) = \sum_{l\,:\,\langle i,l\rangle =1} Q_{il} + \sum_{k\in G_i\cap \F} q_k.
\end{equation}

\noindent Thus

\begin{equation}
f = \sum_{k>l,\langle k,l\rangle =1} \frac{Q^2_{kl}}{\kappa_{kl}} + 2\sum_{i=1}^n p_i \sum_{k\in G_i\cap \F} q_k = D + C
\end{equation}

\noindent where $C$ is constant for any flow field $Q_{kl}$ that satisfies conservation of mass, the prescribed flow boundary condition, and $Q_{kl} = 0$ whenever $\kappa_{kl} = 0$, which is necessary for $f<\infty$ thus necessary for $Q_{kl}$ being a global minimizer of $f$ when $\kappa_{kl}$'s are fixed. So, if $Q_{kl}$ is not already the Kirchhoff flow and we change the current $Q_{kl}$'s to the Kirchhoff flow (the network is physical according to the same argument in Section \ref{sec:proof_no_pressure_path}) then we will decrease the complementary dissipation, which is now equivalent to decreasing dissipation. Thus flow adjusting gives us a network with strictly smaller smaller dissipation, contradicting our assumption that we were starting with a global minimizer.

To complete our proof we must consider the case that we do need to disconnect the last path and this last path has constant flow on it. This path cannot comprise the whole network because if it were to comprise the entire network from the assumption $\F\neq \phi$ we must have $k_i\in \F$ for $1<i<N$ where $N$ denotes the number of vertices in this path (that is all the path vertices between $i=1$ and $i=N$, exclusively, are flow vertices) and there is at least one such interior vertex, or there is an isolated $k\in \F$ that does not connect to any other vertex, which cannot be true for a physical network. Then $\sum_{l\,:\,\langle k_i,l\rangle =1} Q_{k_i l} = 0$, a contradiction to the fact that $k_i$ is a flow vertex. Thus we can disconnect the last path in any case and the argument goes through as before to a contradiction.

\end{proof}

\section{Proof for Proposition \ref{prop:material_formulation_scale_invar}} \label{sec:proof_scale_invar}

\begin{proof}
The material-invariance property of minimally dissipative network under material constraint (\ref{eq:material_constraint_d}) means that all the global minimizers will have the same dissipation if their materials are scaled to be the same. To see this suppose $\{\kappa_{kl}\}, \{\kappa'_{kl}\}$ are minimally dissipative networks with material $K, K'$. Consider the network $\{\beta\kappa_{kl}\}$ with $\beta = (\frac{K'}{K})^2$, so that $\{\beta\kappa_{kl}\}$ has material $K'$. Then since $\{\kappa'_{kl}\}$ is a minimally dissipative network with material $K'$ we have

\begin{equation}
\frac{D}{\beta}\doteq \sum_{k>l,\langle k,l\rangle = 1} \frac{Q_{kl}^2}{\beta \kappa_{kl}} \geq  \sum_{k>l,\langle k,l\rangle = 1} \frac{Q_{kl}^{'2}}{\kappa'_{kl}} = D'.
\end{equation}
Similarly if we define $\beta' = (\frac{K}{K'})^2 = \frac{1}{\beta}$ we have

\begin{equation}
\frac{D'}{\beta'} \geq D \Rightarrow D' = \frac{D}{\beta}
\end{equation}
and the networks $\{\kappa_{kl}\}, \{\kappa'_{kl}\}$ have the same dissipation if their materials are scaled to be the same. This implies if $\{\hat{\kappa}_{kl}\}$ is a minimally dissipative network with material $K=1$, then $\{\beta \hat{\kappa}_{kl}\}$ is a minimally dissipative network of any $K>0$ with $\beta = K^2$. While $\{\hat{\kappa}_{kl}\}$ is not unique since all the minimally dissipative networks with $K=1$ have the same dissipation it does not matter which network we use. Now consider a minimally dissipative network $\{\kappa_{kl}\}$ with material penalty under coefficient $a$ (\ref{eq:constraint_material_penalty}). Suppose this network has material $K$. The network must be a minimally dissipative network with material constraint $K$. If it were not also the minimally dissipative network, then the minimally dissipative network would have a smaller value of $\Theta$ in Equation (\ref{eq:constraint_material_penalty}). Thus we can assume $\kappa_{kl} = \beta\hat{\kappa}_{kl}$ where $\beta = K^2$. A concern is that global minimizers of (\ref{eq:constraint_material_penalty}) under the same coefficient $a$ may have different amounts of material. However since they all have the form $\{\beta\hat{\kappa}_{kl}\}$ for some unit network $\{\hat{\kappa}_{kl}\}$ we can calculate

\begin{equation}
\Theta = \sum_{k>l,\langle k,l\rangle =1} \frac{\hat{Q}^2_{kl}}{\beta \hat{\kappa}_{kl}} + a\sum_{k>l,\langle k,l\rangle =1} \beta^\frac{1}{2}\hat{\kappa}^\frac{1}{2}_{kl}d_{kl} = \frac{\hat{D}}{\beta} + a\beta^\frac{1}{2}.
\end{equation}
If $\{\beta\hat{\kappa}_{kl}\}$ is truly a global minimizer the derivative must vanish since $\Theta\to \infty$ as $\beta\to 0^+,\infty$, i.e.

\begin{equation}
0 = \frac{d\Theta}{d\beta} = -\frac{\hat{D}}{\beta^2} + \frac{a}{2}\beta^{-\frac{1}{2}} \Rightarrow \beta = (\frac{2\hat{D}}{a})^\frac{2}{3}.
\end{equation}
Since $\beta = K^2$ where $K$ is the material of the network $\{\beta\hat{\kappa}_{kl}\}$ we have

\begin{equation}
K = (\frac{2\hat{D}}{a})^\frac{1}{3}
\end{equation}
and in particular, all networks must have the same value of $K$. This bijection $K(a)$ between material constraint and coefficient of material penalty shows that the two different formulations are equivalent for minimally dissipative network under the material-invariance assumption.
\end{proof}

\section{Discussion} \label{sec:discussion}

To summarize our mathematical results we gave a rigorous proof that Murray's law is a necessary condition for global minimization of complementary dissipation (\ref{func:tot_energy}), and showed that it is also necessary for minimally dissipative networks when a flow vertex is present. We proved that under general boundary conditions a global minimizer of complementary dissipation has no loops and does not connect pressure vertices with the same specified pressure. When a flow vertex is present these results recover Durand's previous proof for the no-loop property of minimally dissipative networks. Finally we proved that imposing material as constraint or penalty are equivalent for minimally dissipative networks but is not the case in general. We prove previous results on minimally dissipative networks with mathematical rigor, and extend them to general boundary conditions as well as showing the proper generalization of Rayleigh's and Thomson's hundred year old theorems to include boundary pressures.

Murray's law has shaped understanding of biological transport networks including animals and plants \cite{sherman1981connecting,mcculloh2003water}. However, the derivation of Murray's law has until now been heuristic, ignoring both the coupling between flows and conductances (i.e. assuming flows remain constant while conductances are optimized), and the potential for different boundary conditions on the network \cite{murray1926physiological,durand2007structure}. Our work establishes Murray's law as a necessary condition for networks globally minimizing a complementary dissipation function (\ref{func:tot_energy}), and for minimal dissipative networks under general boundary conditions. As subsidiary steps we remormulated Thomson's principle and Rayleigh's principle for networks with general boundary conditions.

Minimally dissipative networks with flow boundary conditions have been studied both theoretically and numerically\cite{bohn2007structure,durand2007structure,katifori2010damage}. However the effect of pressure boundary conditions upon network structure seems to have received little scrutiny. Imposing pressure rather than flow boundary conditions can be convenient when dealing with complex networks in which only a small part of the entire network may be mapped, for example in high resolution cerebrovascular imaging, which is currently being used to understand the connection between brain function and vascular development or damage\cite{blinder2013cortical}. It may be appropriate to apply pressure boundary condition at the vertices making up periphery of the mapped network. Here monotonicity and boundedness results derived from our extension of Rayleigh's theorem can provide useful estimation tools, and insight into the effect, for example, of adding addition pressure vertices to a cardiovascular network.

Our work is also among the first to elucidate differences between imposing the total material as a constraint or penalty on minimally dissipative networks. Historically Murray derived his law based on a material penalty formulation \cite{murray1926physiological}, but later work treated material as a constraint \cite{bohn2007structure,katifori2010damage,durand2007structure}. Our results show that for minimally dissipative networks these formulations are equivalent, and so recent results are consistent with Murray's original derivation. However the equivalence of the two formulations hinges on two key results: 1. That flows in physical networks minimize complementary dissipation, which is equivalent in tree-networks to minimizing dissipation. 2. Optimal networks are trees. However, these two results can not be appealed to when optimizing other functions on networks. Indeed for general target functions and constraints the formulation one chooses has fundamental effects on the optimal network; as we demonstrated when we optimize flow uniformity, optimal networks may only exist for one formulation and not for the other. Moreover the optimal network may also vary quantitatively as a function of the total allowed material, and possibly with the coefficient of material penalty. Our previous work on optimizing flow uniformity showed evidence of phase transitions as penalty coefficients varied (in preparation). For general functions and constraint one needs to pick analyze the physics and biology carefully to find the appropriate formulation.

Minimal dissipation arguments give theoretical insights in biological network \cite{murray1926physiological}, but are not universal explanatory tools. It has been shown that the leaf vascular network and slime mold network are designed for robustness \cite{katifori2010damage,tero2010rules} and fungus network for mixing \cite{roper2013nuclear}. Moreover even when we seek to minimize dissipation, our function $f$ may be non-newtonian. For example the effective viscosity of blood changes with the cell concentration and with vessel radius \cite{pries2005microvascular}, and it is possible that Murray's law has to be modified in this occasion. The techniques we present in this paper might be generalized to establish the modified Murray's law as a necessary condition for the minimal dissipative networks. Our previous work on zebrafish embryo showed that the uniformity of blood flow is maintained at the cost of dissipation \cite{chang2015optimal}. While numerical algorithms have been designed for finding optimal networks other than minimal dissipation \cite{katifori2010damage}, to our best knowledge there is no theoretical result on the morphology of optimal networks under other functions with either material constraint or penalty. Critically the results presented here draw extensively on monotonicity and boundedness results that cannot be readily generalized to the general case. New methodology is needed to deduce theoretical results for general functions.

\section{Acknowledgments}

This research was funded by grants from the NSF (under grant DMS-1351860). MR. SSC was also supported by the National Institutes of Health, under a Ruth L. Kirschstein National Research Service Award (T32-GM008185). The contents of this paper are solely the responsibility of the authors and do not necessarily represent the official views of the NIH. MR also thanks Eleni Katifori and Karen Alim for useful discussions, and the American Institute of Mathematics for hosting him during one part of the development of this paper.

\appendix\section{Well-posedness of Kirchhoff's laws}

For completeness we give a proof on well-posedness of Kirchhoff's laws. If the network has several connected components we can prove that each component has a unique Kirchhoff flow so without loss of generality we can consider a connected network $G$, i.e. $\forall k,l\in G \; \exists k_1,...,k_n$ s.t. $\langle k_i,k_{i+1}\rangle = 1, \kappa_{k_i k_{i+1}}>0 \; \forall i=1,...,n-1$ and $k_1 = k, k_n = l$, where $\kappa_{kl}$ denotes the conductance of the link $kl$. Now we write down the Kirchhoff system

\begin{equation}
Dp = b
\end{equation}

\noindent where

\begin{equation}
D_{kl} \doteq \left\{\begin{array}{llll}
\sum_{l,\langle k,l\rangle =1} \kappa_{kl} & k=l, k\notin \Pn\\
-\kappa_{kl} & \langle k,l\rangle =1, k\notin \Pn\\
1 & k=l, k\in \Pn\\
0 & o.w.
\end{array}\right.
\label{eq:unifsimp_kappamatrix}
\end{equation}

\noindent and

\begin{equation}
b_k = \left\{\begin{array}{lll}
q_k & k\in \F\\
\bar{p}_k & k\in \Pn\\
0 & o.w.
\end{array}\right..
\end{equation}

\noindent Here the notations follow those in Section \ref{sec:notation}. First we show that if $\Pn \neq \phi$ then $D$ is invertible, which is equivalent to showing that

\begin{equation}
Dp = 0 \Rightarrow p =0.
\label{eq:no_flow_pressure}
\end{equation}

\noindent The solution $p$ for Equation (\ref{eq:no_flow_pressure}) corresponds to a network where we do not have any flows into the system except possibly at vertices with pressure boundary conditions prescribed zero pressures, denoted by $\Pn$. The goal is to show that $p_k = 0 \; \forall k$. Suppose for contradiction that $\exists i\notin \Pn$ s.t. $p_i \neq 0$ (since we already have $p_j = 0 \; \forall j\in \Pn$). Then we would have $Q_{kl} \neq 0$ for some $\langle k,l\rangle =1$ since the network is connected, and without loss of generality let $Q_{kl}>0$. Now we can trace this flow throughout the network in the following procedure:

\begin{enumerate}
\item Given that $Q_{k_{n-1} k_n} >0$ first check if $k_n \in \Pn$, and stop if this is the case.
\item Consider all vertices $l$ s.t. $\langle k_n,l\rangle = 1$. According to Kirchhoff's first law there must be an $l$ s.t. $Q_{k_n l} >0$. Since the network is finite we can pick e.g. the smallest $l$ satisfying these conditions and let $k_{n+1} = l$.
\item Repeat the procedure until $k_N \in \Pn$ for some $N$ and stop. 
\end{enumerate}

If we start with $k_1 = k, k_2 = l$ we can initiate the process since the first condition is satisfied. This procedure has to stop eventually because the network is finite and that $k_1,...,k_n$ are all distinct for any given $n>1$. To see this suppose $k_n = k_m$ with $m>n$. Then we would have $p_n>p_{n+1}>\cdots > p_m = p_n$, a contradiction. Thus we would end up with a chain of distinct vertices $k_1,k_2,...,k_N$ with $\langle k_n,k_{n+1}\rangle =1, Q_{k_n k_{n+1}} > 0 \; \forall n=1,...,N-1$, and $N\in \Pn$. Now we repeat the same procedure just with $k'_1 = l, k'_2 = k$ to trace the flows upstream, and we would end up with another chain $k'_1,k'_2,...,k'_{N'}$ with $\langle k'_n,k'_{n+1}\rangle =1, Q_{k'_n k'_{n+1}} < 0 \; \forall n=1,...,N'-1$, and $N' \in \Pn$. Notice that there is no repetition in the set $\{k_1,...,k_N,k'_1,...,k'_{N'}\}$  since $k_n = k'_m$ would lead to the same contradiction due to loop flow. Now we have a loop flow starting and ending at vertices in $\Pn$, a contradiction since all vertices in $\Pn$ have pressure zero. Therefore we have $p_k = 0 \; \forall 1\leq k \leq M$ and $D$ is invertible. Now suppose $\Pn = \phi$ but $\sum_{k\in \F} q_k = 0$. We want to show that solutions $p$ exist and are determined up to an additive constant, so the flows $Q_{kl}$ are uniquely determined. Notice that if we replace say the last row of $D$ by $e_V$ then $D$ is invertible by the previous argument, so rank$(D) \geq V-1$. Also notice that $\sum_k D_{kl} = 0 \; \forall 1\leq l \leq V$, so rand$(D) = V-1$. Now without loss of generality let vertex $1 \in \F$ and we want to find a solution (if $\F = \phi$ then $p = (0,...,0)^T$ is a solution). If we change the first row of $D$ to $e_1$ it is equivalent to setting $1\in \Pn$ with $\bar{p}_1 = q_1$, which admits a unique solution $p'_k$ by previous argument. Now calculate

\begin{equation}
0 = \sum_{\langle k,l\rangle =1} (p'_k-p'_l)\kappa_{kl} = \sum_k \sum_{l\,:\,\langle k,l\rangle =1} Q'_{kl} = \sum_{l\,:\,\langle 1,l\rangle =1} Q'_{1l} + \sum_{k\in \F, k\neq 1} q_k
\end{equation}

\begin{equation}
\Rightarrow \sum_{l\,:\,\langle 1,l\rangle =1} Q'_{1l} = q_1
\end{equation}

\noindent so $p'_k$ is a solution to the original linear system. By $\sum_k D_{kl} = 0 \; \forall 1\leq l \leq V$ and rand$(D) = V-1$ we know that the null space of $D$ is $\{(a,...,a)^T|a\in \R\}$, so the general solution is

\begin{equation}
p_k = p'_k + a \quad \forall 1\leq k \leq V
\end{equation}

\noindent for every $a\in \R$. Thus $p$ is determined up to a constant and the flow is uniquely determined.

Finally we show that there is no solution of $p$ when $\sum_{k\in \F} q_k \neq 0$. This is straight-forward since suppose for contradiction that $\exists p \in \R^V$ s.t.

\begin{equation}
Dp = b.
\end{equation}

\noindent Then if we multiply both side from the left by $(1,...,1)$ then since $(1,...,1)D = 0$ we gent

\begin{equation}
0 = \sum_{k\in \F} q_k,
\end{equation}

\noindent a contradiction.

\end{document}